\documentclass[12pt]{article}

\usepackage[utf8]{inputenc}

\usepackage[a4paper,nomarginpar]{geometry}
\usepackage[english]{babel}

\usepackage{amsmath,amsfonts,amssymb,amsthm,amscd}
\usepackage[pdftex]{graphicx}

\AtBeginDocument{
   \def\MR#1{}
}

\theoremstyle{definition}
\newtheorem{theorem}{Theorem}

\newtheorem{proposition}[theorem]{Proposition}
\newtheorem{corollary}[theorem]{Corollary}

\theoremstyle{definition}
\newtheorem{definition}[theorem]{Definition}
\newtheorem{example}[theorem]{Example}

\theoremstyle{remark}
\newtheorem{remark}[theorem]{Remark}

\newcommand{\N}{\mathbb{N}} 
\newcommand{\R}{\mathbb{R}} 
\newcommand{\C}{\mathbb{C}} 

\newcommand{\norm}[1]{\lVert#1\rVert}

\newcommand{\lDens}{\underline{\operatorname{Dens}}}
\newcommand{\uDens}{\overline{\operatorname{Dens}}}

\begin{document}

\title{The Specification Property for $C_0$-Semigroups}

\author{
S. Bartoll, F. Mart\'{\i}nez-Gim\'{e}nez, A. Peris\footnote{Corresponding author. E-mail: aperis@mat.upv.es} and F. Rodenas }

\date{ }

\maketitle

\begin{center}
 Institut Universitari de Matemàtica Pura i Aplicada, Universitat Polit\`ecnica de Val\`encia, Edifici 8E, Acces F, 4a planta, 46022 Val\`encia, Spain.
\end{center}

\begin{abstract}
We study one of the strongest versions of chaos for continuous dynamical systems, namely the specification property.
We extend the definition of specification property for operators on a Banach space to  strongly continuous one-parameter semigroups of operators, that is,
$C_0$-semigroups. In addition, we study  the relationships of the specification property  for $C_0$-semigroups (SgSP) with other dynamical properties: mixing,
Devaney's chaos, distributional chaos and frequent hypercyclicity. Concerning the applications, we provide several examples of semigroups which exhibit the SgSP with particular interest on solution semigroups to certain linear PDEs, which range from the hyperbolic heat equation to the Black-Scholes equation.

\end{abstract}


\section{Introduction}

A (continuous) map on a metric space satisfies the specification property (SP) if for any choice of points,  one can approximate distinct pieces of orbits by a
single periodic orbit with a certain uniformity. It was first introduced by Bowen~\cite{bowen1970topological}; since then, this property has attracted the
interest of many researchers (see, for instance,  the early work \cite{sigmund1974on}). In a few words, the specification property requires that, for a given distance $\delta>0$, and for any finite family of points, there is always a periodic orbit that traces arbitrary long pieces of the orbits of the family, up to a distance $\delta$, allowing a minimum ``shift time'' $N_\delta$ from one piece of orbit to another one, which only depends on $\delta$.

\begin{definition}
A continuous map $f:X\to X$ on a compact metric space $(X,d)$ has the  specification property  if for any $\delta>0$ there is a positive integer $N_\delta$ such
that for any integer $s\ge2$, any set $\{y_1,\dots,y_s\}\subset X$ and any integers $0=j_1 \le k_1 < j_2 \le k_2 < \dots < j_s \le k_s$ satisfying $j_{r+1} - k_r \ge N_\delta$ for $r = 1, \ldots, s-1$, there is a point $x\in X$ such that the following conditions hold:
\begin{gather*}
 d(f^i(x),f^i(y_r)) < \delta,  \; \mbox{with} \; j_r \le i \le k_r , \, \mbox{ for every} \; r \le s,
 \\
 f^{N_\delta + j_s}(x) = x.
\end{gather*}
\end{definition}

This definition must be modified  when one deals with  bounded linear operators defined on separable Banach spaces which are never compact
\cite{bartoll_martinez_peris2012the,bartoll_martinez_peris2016operators}.

\begin{definition}\label{OSP}
An operator $T:X\to X$ on a separable Banach space $X$ has the \emph{operator specification property} (OSP) if there exists an increasing sequence $(K_m)_m$ of $T$-invariant sets  with $0\in K_1$ and $\overline{\cup_{m\in\N} K_m}=X$ such that for each $m\in \N$ the map $T|_{K_m}$ has the specification property, that is, for any $\delta>0$ there is a positive integer $N_{\delta,m}$ such that for every $s\ge2$, any set $\{y_1,\dots,y_s\}\subset K_m$ and any integers $0=j_1 \le k_1 < j_2 \le k_2 < \dots < j_s \le k_s$ with $j_{r+1} - k_r \ge N_{\delta,m}$ for $1\le r \le s-1$, there is a point $x\in K_m$ such that, for each positive integer $r \le s$  and any integer $i$ with $j_r \le i \le k_r$, the following conditions hold:
\begin{gather*}
 \norm{T^i(x)-T^i(y_r)} < \delta,
 \\
 T^n(x) = x, \mbox{ where } n = N_{\delta,m} + k_s.
\end{gather*}
\end{definition}

 We also recall that a continuous map on a metric space is said to be
\emph{chaotic} in the sense of Devaney if it is topologically transitive and the set of periodic points is dense. Although there is no common agreement about what a
chaotic map is, the specification property  is stronger than  Devaney's definition of chaos. Recently, several properties of linear operators with the OSP and
the connections of this OSP with other well-known dynamical properties, like mixing, chaos in the sense of Devaney and frequent hypercyclicity have been studied
in \cite{bartoll_martinez_peris2016operators}, we will use these results throughout the paper. Other recent works on the specification property are
\cite{lampart_oprocha2009shift,oprocha2007specification,oprocha_stefankova2008specification}.

A family $(T_t)_{t\ge 0}$ of linear and continuous operators on a Banach space $X$ is said to be a \textit{$C_0$-semigroup} if $T_0=Id$, $T_tT_s=T_{t+s}$ for
all $t,s\ge 0$, and $\lim_{t\rightarrow s}T_tx=T_sx$ for all $x\in X$ and $s\ge 0$.

Let $(T_t)_{t\geq 0}$ be an arbitrary $C_0$-semigroup on $X$. In that case, the (linear) map defined by $Ax:=\lim_{t\rightarrow 0} \frac{1}{t} (T_tx-x)$
exists on a dense subspace of $X$; denoted by $D(A)$. Then $A$, or rather $(A,D(A))$, is called the \textit{(infinitesimal) generator} of the semigroup. It is also known that the infinitesimal generator determines the semigroup uniquely. If the generator $A$ is defined on $X$ ($D(A)=X$), the semigroup is expressed
as  $\{T_t \}_{t\geq 0}=\{e^{tA}\}_{t\geq 0}$, and it is called a \emph{uniformly continuous} semigroup.

Given a family of operators $(T_t)_{t\ge 0}$, we say it is \emph{transitive} if for every pair of non-empty open sets $U,V\subset X$
there exists some $t>0$ such that $T_t(U)\cap V\ne \emptyset$. Furthermore, if there is some $t_0$ such that the condition $T_t(U)\cap V\ne \emptyset$ holds for
every $t\ge t_0$ we say that it is \emph{topologically mixing} or  mixing.

A family of operators $(T_t)_{t\ge 0}$ is said to be \emph{universal} if there exists some $x\in X$ such that $\{T_tx\,:\,t\ge 0\}$ is dense in $X$. When
$(T_t)_{t\ge 0}$ is a $C_0$-semigroup we particularly refer to it as \emph{hypercyclic}. In this setting, transitivity coincides with universality, but
it is strictly weaker than mixing \cite{bermudez_bonilla_conejero_peris2005hypercyclic}.

In addition, two notions of chaos are recalled: \emph{Devaney chaos} and \emph{distributional chaos}. An element $x\in X$ is said to be
a \textit{periodic point} of $(T_t)_{t\ge 0}$ if there exists some $t_0>0$ such that $T_{t_0}x=x$. A  family of operators $(T_t)_{t\ge 0}$  is said to be
\emph{chaotic in the sense of Devaney} if it is hypercyclic   and there exists a dense set of periodic points in $X$. On the other hand, it is
\emph{distributionally chaotic} if there are an uncountable set $S\subset X$ and $\delta>0$, so that for each $\varepsilon>0$ and each pair $x,y\in S$ of
distinct points we have

$$\uDens\{s\ge 0:||T_sx -T_sy||\ge \delta\}=1 \ \mbox{ and } \ \uDens\{s \ge 0 : ||T_sx - T_sy|| < \varepsilon\} = 1,$$

\noindent where $\uDens(B)$ is the upper density of a Lebesgue measurable subset $B\subset \mathbb{R}_0^+$ defined as

$$\limsup_{t\rightarrow\infty}\frac{\mu(B\cap[0,t])}{t},$$

\noindent with $\mu$ standing for the Lebesgue measure on $\mathbb{R}_0^+$.
A vector $x\in X$ is said to be \emph{distributionally irregular} for the $C_0$-semigroup $(T_t)_{t\ge 0}$ if for every $\varepsilon>0$ we have

$$\uDens\{s\ge 0:||T_sx||\ge \varepsilon^{-1}\}=1  \ \mbox{ and } \ \uDens\{s \ge 0 : ||T_sx|| < \varepsilon\} = 1.$$  

Such vectors were considered in \cite{bernardes_bonilla_muller_peris2013distributional} so as to get a further insight into the phenomenon of distributional
chaos, showing the equivalence between a distributionally chaotic operator and an operator having a distributionally irregular vector. This equivalence was later generalized for $C_0$-semigroups in \cite{albanese_barrachina_mangino_peris2013distributionally}.

Devaney chaos, hypercyclicity and  mixing properties have been widely studied for linear operators on Banach and more general spaces
\cite{bermudez_bonilla_conejero_peris2005hypercyclic,bonet_martinez-gimenez_peris2001a,costakis_sambarino2004topologically,godefroy_shapiro1991operators,grivaux2005hypercyclic,peris_saldivia2005syndetically}.
The recent books \cite{bayart_matheron2009dynamics} and \cite{grosse-erdmann_peris2011linear} contain
the basic theory, examples, and many results on chaotic linear dynamics. Particular attention deserves the case of $C_0$-semigroups of operators, since many of them are originated in the analysis of the asymptotic behaviour of solutions to certain linear partial differential equations and to infinite systems of linear differential equations. Especially, different notions of chaos for $C_0$-semigroups have experienced a great development in recent years (see, e.g., \cite{albanese_barrachina_mangino_peris2013distributionally,aroza_kalmes_mangino2014chaotic,banasiak_moszynski2011dynamics,
bernardes_bonilla_muller_peris2015li,chakir_el-mourchid2018strong, 
conejero_lizama_murillo-arcila2017chaotic,conejero_lizama_murillo-arcila_peris2017linear,
conejero_mangino2010hypercyclic,conejero_peris2009hypercyclic,costakis_peris2002hypercyclic,kalmes2010hypercyclicity,
rudnicki2012chaoticity,yin_wei2018recurrence}).

A stronger concept than hypercyclicity is the notion  of frequent hypercyclicity introduced by Bayart and Grivaux
\cite{bayart_grivaux2006frequently} (see \cite{grosse-erdmann_peris2011linear} and the references therein) trying to quantify the frequency with which an
orbit meets an open set. This concept was extended to $C_0$-semigroups in \cite{badea_grivaux2007unimodular}.

A $C_0$-semigroup $(T_t)_{t\ge 0}$ is said to be \emph{frequently hypercyclic} if there exists $x \in X$ (called frequently hypercyclic vector) such that
$\lDens( \{ t \geq 0: T_tx \in U\})>0$ for every non-empty open subset $U \subset X$, where $\lDens(B)$ is the lower density of a Lebesgue measurable subset
$B\subset \mathbb{R}_0^+$ defined as

$$\liminf_{t\rightarrow\infty}\frac{\mu(B\cap[0,t])}{t}.$$

\noindent In \cite{conejero_muller_peris2007hypercyclic} it was proved that if $x \in X$ is  a (frequently) hypercyclic vector for  $(T_t)_{t\ge 0}$, then  $x$ is a
(frequently) hypercyclic vector for every operator $T_t$, $t > 0$.

In \cite{bonilla_grosse-erdmann2007frequently}  Bonilla and Grosse-Erdmann, based on a result of Bayart and Grivaux, provided a Frequent Hypercyclicity Criterion
for operators. Later, Mangino and Peris  \cite{mangino_peris2011frequently} obtained a continuous version of the criterion based on the Pettis integral, which is
called the Frequent Hypercyclicity Criterion for semigroups.

The aim of this work is to study  the specification property  for strongly continuous semigroups of operators on Banach spaces, that is, for $C_0$-semigroups
and its relationship with other dynamical properties, like hypercyclicity, mixing, chaos and frequent hypercyclicity; and to provide useful tools ensuring that many natural solution semigroups associated to linear PDEs satisfy the specification property for $C_0$-semigroups. The paper is structured as follows:  In
Section~\ref{sec_ssp} we introduce the notion of the specification property  for $C_0$-semigroups, from now on denoted by SgSP. Section~\ref{sec_sspp} is devoted to study the SgSP in connection
with other dynamical properties. Finally, in Section~\ref{sec_ex} we provide several applications of the results in previous sections to solution semigroups of certain linear PDEs, and a characterization of translation semigroups which exhibit the SgSP.

\section{Specification property for $C_0$-semigroups}\label{sec_ssp}

A first notion of the specification property for a one-parameter family of continuous maps acting on a compact metric space was given in \cite{bowen1972periodic}. When trying to study the specification property in the context of semigroups of linear operators defined on separable Banach spaces, the first
crucial problem is that these spaces are never compact, therefore, our first task should be to adjust the SP in this context, in the vain of the discrete case, and the following
definition can be considered the natural extension in this setting.

\begin{definition}[Specification property for semigroups, SgSP]\label{SP-semigroup}\label{weakSP-semigroup}
A $C_0$-semigroup $(T_t)_{t\ge 0}$ on a separable Banach space $X$ has the SgSP if there exists an increasing sequence $(K_n)_n$ of  sets with
$0\in K_1$ and $\overline{\cup_{n\in\N} K_n}=X$ and  there exists a $t_0\geq 0$, such that for each $n\in \N$ the set $K_n$ is $T_{t_0}$-invariant, and  for any $\delta>0$ there is
$M_{\delta,n} >t_0$ such that for any integer $s\ge2$, any set $\{y_1,\dots,y_s\} \subset K_n$ and any real numbers: $0=a_1 \le b_1 < a_2 \le b_2 <
\dots < a_s \le b_s$ satisfying  $a_{r+1} - b_r \ge M_{\delta,n}$ for $r = 1, \ldots, s-1$, there is a point
$x\in K_n$ such that the following conditions hold:
\begin{gather*}
 \norm{T_{t}(x)-T_{t}(y_r)} < \delta, \ \ \mbox{ for each  } t \in [a_r,b_r], \ \ r=1,2,\dots ,s,
 \\
 T_{t'}(x) = x, \ \mbox{ where } t' \in [M_{\delta,n} + b_s, M_{\delta,n} + b_s+t_0[\cap \N\cdot t_0 .
\end{gather*}
\end{definition}

Analogously to the discrete case, the meaning of this property is that if the semigroup has the SgSP then it is possible to  approximate simultaneously several
finite pieces of orbits by one periodic orbit. Obviously, parameter intervals for the approximations must be disjoint. In contrast, the periodicity condition SgSP looks
more relaxed than the corresponding one in the OSP: Observe that in the SgSP we have the existence of certain $t_0>0$ such that the periodic vectors $x$ that trace the orbits are
such that $T_t(x)=x$ for some integer multiple $t$ of $t_0$. The reason to relax the requirements on the periods in the SgSP is that almost all kind of behaviours concerning periods
in chaotic $C_0$-semigroups can occur, as Bayart and Bermúdez showed in \cite{bayart_bermudez2009semigroups}.

\section{The semigroup specification property and other dynamical properties for $C_0$-semigroups}\label{sec_sspp}

In this section, we study the relationship between the specification property and topological mixing, Devaney chaos, distributional chaos and frequent hypercyclicity.

We first need the following result that connects the discrete case with the continuous case concerning the specification property.

\begin{proposition}\label{SP2-operators}
Let $(T_t)_{t\geq 0}$ be a $C_0$-semigroup on a separable Banach space $X$. Then the following assertions are equivalent:
\begin{enumerate}
\item[(a)] $(T_t)_{t\geq 0}$ has the SgSP,
\item[(b)] some operator  $T_{t_0}$ has the OSP.
\end{enumerate}
\end{proposition}

\begin{proof} The implication (a)$\to$(b) is immediate by the definitions of OSP and SgSP.

Suppose now that $T_{t_0}$ has the OSP for some $t_0>0$. We find then an increasing sequence $(K_m)_m$ of $T_{t_0}$-invariant sets  with $0\in K_1$ and $\overline{\cup_{m\in\N} K_m}=X$ such that for each $m\in \N$ and for any $\delta>0$ there is a positive integer $N_{\delta,m}$ satisfying the conditions of the OSP for $T_{t_0}$.  Given $n\in \N$ and $\delta>0$, let $C:=\max_{t\in [0,t_0]} \norm{T_t}$ and $\delta'>0$ with $\delta'<\delta/(1+C)$. We fix $M_{\delta,n} =(N_{\delta',n}+2)t_0$. For any integer $s\ge2$, any set $\{y_1,\dots,y_s\} \subset K_n$ and any real numbers: $0=a_1 \le b_1 < a_2 \le b_2 <
\dots < a_s \le b_s$ satisfying  $a_{r+1} - b_r \ge M_{\delta,n}$ for $r = 1, \ldots, s-1$, we have to show the existence of a point
$x\in K_n$ such that the following conditions hold:
\begin{gather*}
 \norm{T_{t}(x)-T_{t}(y_r)} < \delta, \ \ \mbox{ for each  } t \in [a_r,b_r], \ \ r=1,2,\dots ,s,
 \\
 T_{t'}(x) = x, \ \mbox{ where } t' \in [M_{\delta,n} + b_s, M_{\delta,n} + b_s+t_0[\cap \N\cdot t_0 .
\end{gather*}

We select integers $0=j_1 \le k_1 < j_2 \le k_2 < \dots < j_s \le k_s$ with $[a_r,b_r]\subset [j_rt_0,k_rt_0]\subset [a_r-t_0,b_r+t_0]$, $r=1,\dots ,s$. We then have that $j_{r+1} - k_r \ge N_{\delta',n}$ for $1\le r \le s-1$. By hypothesis there is a point $x\in K_n$ such that, for each positive integer $r \le s$  and any integer $i$ with $j_r \le i \le k_r$, the following conditions hold:
\begin{gather*}
 \norm{T_{t_0}^i(x)-T_{t_0}^i(y_r)} < \delta',
 \\
 T_{t_0}^n(x) = x, \mbox{ where } n = N_{\delta',m} + k_s.
\end{gather*}

If $t \in [a_r,b_r]$, there is an integer $i\in [j_r,k_r]$ such that $t-t_0<i\cdot t_0\leq t$. Thus,
$$
\norm{T_{t}(x)-T_{t}(y_r)}\leq \norm{T_{t-it_0}}\norm{T_{it_0}(x)-T_{it_0}(y_r)}\leq C  \norm{T_{t_0}^i(x)-T_{t_0}^i(y_r)} < C\delta'<\delta .
$$
Finally, since $T_{nt_0}(x)=x$ for $n = N_{\delta',m} + k_s$, we then conclude that $T_{t'}(x) = x$ for some $t' \in [M_{\delta,n} + b_s, M_{\delta,n} + b_s+t_0[\cap \N\cdot t_0$.
\end{proof}

The following observation is useful to characterize mixing semigroups (see \cite{grosse-erdmann_peris2011linear}).

\begin{remark}
Let $(T_t)_{t\geq 0}$ be a $C_0$-semigroup on a separable Banach space $X$.  The semigroup $(T_t)_{t\geq 0}$ is  mixing if and only if for every non-empty open
set $U \subset X$ and every open 0-neighbourhood $W$, the return sets $R(U,W)$ and $R(W,U)$ are  co-bounded, where $R(A,B):=\{ t\geq 0 \ ; \ T_t(A)\cap B\not\emptyset\}$.
\end{remark}

\begin{proposition}\label{SP-mixing}
Let $(T_t)_{t\geq 0}$ be a $C_0$-semigroup on a separable Banach space $X$. If  $(T_t)_{t\ge 0}$ has the SgSP, then $(T_t)_{t\ge 0}$ is mixing.
\end{proposition}

\begin{proof}
 Let us consider a non-empty open set $U$ and a 0-neighbourhood $W$.  We claim that there exists  some $t_1 >0$ such that
$t \in R(U,W) \cap R(W,U), \, \forall t>t_1$,  and this implies $(T_t)_{t\geq 0}$ is mixing.

 Fix $u\in U$ and $\delta>0$ such that $B(u,2\delta)\subset U$ and $B(0,2\delta)\subset W$.
 By hypothesis, $(T_t)_{t\geq 0}$ has the SgSP, then there are $t_0>0$ and  a $T_{t_0}$-invariant set $K$ such that the restriction of $T_{t_0}$ to $K$  has the specification property and $K
 \cap B(u,\delta)\ne\emptyset$.
 From Definition \ref{SP-semigroup}, there exists $M$ (depending on $K$ and $\delta$, which we suppose $M\in\N\cdot t_0$) such that if we choose $y_1\in K\cap
 B(u,\delta)$, $y_2=0$, $s>0$ with $s\in \N\cdot t_0$,
 and $0=a_1=b_1<a_2=M <b_2=M+s$ then there exists a periodic point  $x\in K$ with period $2M + s$ such that
\begin{gather*}
 \norm{T_t (x)-T_t (y_1)}<\delta, \ a_1 \leq t \leq b_1 \,,\\
 \norm{T_t (x)-T_t (y_2)} <\delta, \ a_2 \leq t \leq b_2 \,.
 \end{gather*}
This implies $\norm{x-y_1}<\delta$, so $\norm{x-u}<2\delta$ and hence $x\in U$. From the second previous inequality, we have  $T_t (x) \in B(0,\delta)
\subset W$ for $ M \leq t \leq M+s$. Therefore $t \in R(U,W)$ for any $t \geq M$.

Taking now $t>M$, we select $t'\in [M,M+s]$ such that $t+t'\in \N\cdot (2M+s)$. We have $\norm{T_{t'}(x)}<\delta$, hence $T_{t'}(x)\in B(0,\delta)\subset W$. Since $x$ is periodic with period dividing
$2M+s$, then
$T^{t}(T^{t'}(x))=x\in U$. Therefore $t \in R(W,U)$ for any $t > M$.

We have proved that the complementary of $R(U,W) \cap R(W,U)$ is bounded and this finishes the proof.

\end{proof}

\begin{proposition}\label{SP-chaos}
Let $(T_t)_{t\geq 0}$ be a $C_0$-semigroup on a separable Banach space $X$. If  $(T_t)_{t\ge 0}$ has the SgSP then $(T_t)_{t\ge 0}$ is Devaney chaotic.
\end{proposition}

\begin{proof}
By Proposition \ref{SP-mixing},  $(T_t)_{t\geq 0}$ is topologically transitive and, by the definition of SgSP, it is clear that the set of periodic points for the semigroup is dense in $X$.
\end{proof}

\begin{proposition}\label{SP-distributionally}
Let $(T_t)_{t\geq 0}$ be a $C_0$-semigroup on a separable Banach space $X$. If  $(T_t)_{t\ge 0}$ has the SgSP with respect to an increasing sequence $(K_n)_n$ of invariant compact sets, then $(T_t)_{t\ge 0}$ is distributionally
chaotic.
\end{proposition}

\begin{proof}
We first recall that for single maps on compact metric spaces, Oprocha \cite{oprocha2007specification} showed that the SP implies distributional chaos in our sense. Since there is $t_0>0$ such that $T_{t_0}$ has the OSP, and by hypothesis the associated increasing sequence $(K_n)_n$ of invariant sets consists of compact sets, then $T_{t_0}|_{K_n}$ is distributionally chaotic for every $n\in\N$, thus the operator $T_{t_0}$ is distributionally chaotic. Applying Theorem~3.1 in \cite{albanese_barrachina_mangino_peris2013distributionally} we obtain that the semigroup is distributionally chaotic.
\end{proof}

It is well-known \cite{grosse-erdmann_peris2011linear, conejero_muller_peris2007hypercyclic, mangino_peris2011frequently} that a $C_0$-semigroup is hypercyclic
(respectively, mixing, Devaney chaotic, frequently hypercyclic)  if and only if it admits a hypercyclic  (respectively, mixing, Devaney chaotic, frequently hypercyclic) discretization $(T_{t_n})_n$. In particular, it is useful for our purposes the following characterization of frequent hypercyclicity for
semigroups in  terms of the frequent hypercyclicity of some of its operators \cite{conejero_muller_peris2007hypercyclic,mangino_peris2011frequently}.

 \begin{proposition}[\cite{conejero_muller_peris2007hypercyclic, mangino_peris2011frequently}]\label{mangino_peris}
 Let $(T_t)_{t\geq 0}$ be a $C_0$-semigroup on a separable Banach
space $X$. Then the following conditions are equivalent:

(i) $(T_t)_{t\geq 0}$ is frequently hypercyclic.

(ii) For every $t > 0$ the operator $T_t$ is frequently hypercyclic.

(iii) There exists $t_0 > 0$ such that $T_{t_0}$ is frequently hypercyclic.
\end{proposition}

The implication (i)$\rightarrow$(ii) was proved in \cite{conejero_muller_peris2007hypercyclic}, (ii)$\rightarrow$(iii) is obvious, and (iii)$\rightarrow$(i) was observed in
\cite{mangino_peris2011frequently}.

We point out the connection between the frequent hypercyclicity property for semigroups and the SgSP.

\begin{proposition}\label{SP->FH}
Let $(T_t)_{t\geq 0}$ be a $C_0$-semigroup on a separable Banach space $X$. If $(T_t)_{t\ge 0}$ has the SgSP, then  $(T_t)_{t\ge 0}$ is frequently
hypercyclic.
\end{proposition}

\begin{proof}
By Proposition~\ref{SP2-operators}, if $(T_t)_{t\ge 0}$ has the SgSP, then there exists $t_0 > 0$ such that the operator $T_{t_0}$  has the OSP, then the
operator  $T_{t_0}$ is frequently hypercyclic \cite{bartoll_martinez_peris2016operators} and, therefore, the $C_0$-semigroup  $(T_t)_{t\ge 0}$ is frequently hypercyclic (see Proposition~\ref{mangino_peris} \cite{conejero_muller_peris2007hypercyclic,mangino_peris2011frequently}).
\end{proof}

The following result is a partial converse.

\begin{proposition}\label{FHC->SP2}
Let $(T_t)_{t\geq 0}$ be a $C_0$-semigroup on a separable Banach space $X$. If $(T_t)_{t\ge 0}$ satisfies the Frequent  Hypercyclicity Criterion for semigroups of
\cite{mangino_peris2011frequently}, then every operator $T_{t}$, $t>0$, has the OSP and, therefore, the  semigroup $(T_t)_{t\ge 0}$ has the SgSP.
\end{proposition}

\begin{proof}
If $(T_t)_{t\ge 0}$ satisfies the Frequent  Hypercyclicity Criterion for semigroups of \cite{mangino_peris2011frequently} then every operator  $T_{t}$, $t >0$, satisfies the Frequent  Hypercyclicity Criterion for operators of \cite{bonilla_grosse-erdmann2007frequently}. By
\cite[Theorem 14]{bartoll_martinez_peris2016operators}, $T_{t}$
has the OSP for every $t > 0$, and finally,  by using Proposition~\ref{SP2-operators}, we conclude that the semigroup $(T_t)_{t\geq 0}$ has the SgSP.
\end{proof}

Corollary 2.3 in \cite{mangino_peris2011frequently} showed that under some conditions, expressed in terms of eigenvector fields for the infinitesimal
generator $A$ of the $C_0$-semigroup $(T_t)_{t\ge 0}$, the semigroup is frequently hypercyclic. More precisely, it was proved in \cite{mangino_peris2011frequently} that
it satisfies the Frequent  Hypercyclicity Criterion. Actually, the conditions on the infinitesimal generator are much easier to verify on precise examples and applications than the Frequent Hypercyclicity Criterion itself.

\begin{proposition}\label{Agenerator=>SP2} Let $(T_t)_{t\ge 0}$ be a $C_0$-semigroup on a separable complex Banach space $X$ and let
$A$ be its infinitesimal generator.
 Assume that there exists a family $(f_j)_{j\in \Gamma }$ of locally bounded measurable maps
 $f_j:I_j \rightarrow X$ such that $I_j$ is an interval in $\mathbb{R}$, $f_j(I_j)\subset D(A)$, where $D(A)$ denotes the domain of the generator,
                    $Af_j(t)=itf_j(t)$ for every $t \in I_j$, $j\in \Gamma$ and
                     $\mbox{span} \{f_j (t) \ : \ j\in \Gamma,\ t\in I_j\}$ is dense in $X$.
If either

a) $f_j \in C^2(I_j, X)$, $j\in \Gamma$,\\ or

b) $X$ does not contain $c_0$  and $\langle \varphi,f_j\rangle \in C^{1}(I_j)$, $\varphi\in X'$, $j\in\Gamma$,\\ then the semigroup $(T_t)_{t\geq 0}$ has the
SgSP.
\end{proposition}

\begin{proof}
The result  directly follows from the Corollary 2.3 in \cite{mangino_peris2011frequently} and Proposition \ref{FHC->SP2}.
\end{proof}

It was pointed out in  \cite[Remark 2.4]{mangino_peris2011frequently} that the spectral criterion  for chaos of Desch, Schappacher and Webb \cite{desch-schappacher-webb1997hypercyclic}
for $C_0$-semigroups  has stronger requirements than the criterion given in Proposition~\ref{Agenerator=>SP2}. As a consequence, we can obtain a very useful criterion for SgSP.

\begin{corollary}\label{generatorDSW}
Assume that there exists an open connected subset $U\subset \C$ and weakly holomorphic functions $f_j: U\rightarrow X$, $j\in J$, such that
\begin{enumerate}
                    \item $U\cap i\R\not= \emptyset$,
                    \item $f_j(\lambda)\in \ker (\lambda I-A)$ for every $\lambda \in U$, $j\in J$,
                    \item for any $x^\ast\in X^\ast$, if $\langle f_j(\lambda),x^\ast\rangle=0$ for all $\lambda\in U$ and $j\in J$ then $x^\ast=0$.
\end{enumerate}
Then the semigroup $(T_t)_{t\geq 0}$ has the SgSP.
\end{corollary}

\section{Applications and examples}\label{sec_ex}

In this section we will present several examples of $C_0$-semigroups exhibiting the specification property, with particular interest in solution semigroups to certain PDEs. A characterization of translation semigroups with the SgSP is also provided.

In the following examples, in order to ensure whether the solution semigroup has the SgSP, we will check the conditions of Proposition~\ref{Agenerator=>SP2} (\emph{i.e.}, the conditions of Corollary 2.3 in \cite{mangino_peris2011frequently}) or the spectral criterion in
\cite{desch-schappacher-webb1997hypercyclic} for chaos (Corollary~\ref{generatorDSW}).

\begin{example}[The solution semigroup of the hyperbolic heat transfer equation]

 Let us consider  the hyperbolic heat transfer equation (HHTE):
$$
    \left\{
        \begin{array}{ll}
            \tau\frac{\partial^2u}{\partial t^2}+\frac{\partial u}{\partial t}=\alpha\frac{\partial^2u}{\partial x^2},\\\\
            u(0,x)=\varphi_1(x),   x\in\mathbb{R},\\\\
            \frac{\partial u}{\partial t}(0,x)=\varphi_2(x),  x\in\mathbb{R}
        \end{array}
    \right.
$$
 where $\varphi_1$ and $\varphi_2$ represent the initial temperature and the initial variation of temperature, respectively, $\alpha>0$ is the thermal
 diffusivity, and $\tau>0$ is the thermal relaxation time.

 The dynamical behaviour presented by the solutions of the classical heat equation was studied by Herzog \cite{herzog1997on} on certain spaces of analytic
 functions with certain growth control. Later, the dynamical properties of the solution semigroup for the hyperbolic heat transfer equation were also
 established in \cite{conejero_peris_trujillo2010chaotic,grosse-erdmann_peris2011linear}.

The HHTE can be expressed as a  first-order equation on the product of a certain function space with itself $X \oplus X$. We set $u_1=u$ and $u_2=\frac{\partial
u}{\partial t}$. Then the associated first-order equation is: $$
\left\{
\begin{array}{c}
\frac{\partial}{\partial t}\left(
    \begin{array}{c}
      u_1\\
      u_{2} \\
    \end{array}
  \right)=\left(
            \begin{array}{cc}
              0 & I \\
              \frac{\alpha}{\tau}\frac{\partial^2}{\partial x^2} & \frac{-1}{\tau}I \\
            \end{array}
          \right)\left(
                   \begin{array}{c}
                     u_{1} \\
                     u_2 \\
                   \end{array}
                 \right)\\\\
  \left(
    \begin{array}{c}
      u_1(0,x) \\
      u_2(0,x) \\
    \end{array}
  \right)=  \left(
    \begin{array}{c}
      \varphi_1(x) \\
      \varphi_2(x)\\
    \end{array}
  \right)   , \  x\in\mathbb{R}
 \end{array}
    \right.
$$
 We fix $\rho>0$ and consider the space (see \cite{herzog1997on})
 $$X_\rho=\{f:\mathbb{R}\rightarrow \mathbb{C}; f(x)=\sum_{n=0}^\infty\frac{a_n\rho^n}{n!}x^n, (a_n)_{n\geq0}\in c_0\}$$ endowed with the norm
 $||f||=\sup_{n\geq0}|a_n|$, where $c_0$ is the Banach space of complex sequences tending to $0$. Since
$$ 
A:=\left(
    \begin{array}{cc}
        0 & I \\
        \frac{\alpha}{\tau}\frac{\partial^2}{\partial x^2} & \frac{-1}{\tau}I \\
    \end{array}
    \right) 
$$ 
is an operator on $X:=X_\rho\oplus X_\rho$, we have that $(T_t)_{t\geq 0}=(e^{tA})_{t\geq0}$ is the $C_0$-semigroup solution of the HHTE.
  We know from \cite{conejero_peris_trujillo2010chaotic} and \cite{grosse-erdmann_peris2011linear} that, given $\alpha$, $\tau$ and $\rho$ such that
  $\alpha\tau\rho>2$, the solution semigroup $(e^{tA})_{t\geq0}$  defined on $X_\rho\oplus X_\rho$ is mixing and chaotic since it satisfies the hypothesis of
  the spectral  criterion (Corollary~\ref{generatorDSW}). Therefore, it  follows that the solution semigroup of the HHTE has the SgSP.
\end{example}

\begin{remark}
With minor changes, we can apply the previous argument to the wave equation $$\label{cauchy_problem}
\left\{\begin{array}{rl}
       \frac{\partial^2 u}{\partial t^2} &= \alpha \frac{\partial^2 u}{\partial x^2}\\
      u(0,x) & = \varphi_1(x),\quad x\in\mathbb{R}\\
      u_t(0,x) & = \varphi_2(x),\quad x\in\mathbb{R}
\end{array}
\right\}
$$ which can be expressed as  a first order equation in $X_\rho \oplus X_\rho$ (see \cite{grosse-erdmann_peris2011linear}), in order to state that its semigroup
solution has the SgSP.
\end{remark}

\begin{example}[$C_0$-semigroup solution of the Black-Scholes equation]

Black and Scholes proposed in \cite{blackscholes73} a mathematical model which gives a theoretical estimate of the price of stock options. The model is based on
a partial differential equation,  called the Black-Scholes equation, which estimates the price of the option over time. They proved that under some
assumptions about the market, the value of a stock option $u(x,t)$, as a function of the current value of the underlying asset $x\in\mathbb{R}^+=[0,\infty)$ and
time, satisfies the final value problem: 
$$
\left\{\begin{array}{ll}
\frac{\partial u}{\partial t}=-\frac{1}{2}\sigma^2x^2\frac{\partial^2u}{\partial x^2}-rx\frac{\partial u}{\partial x}+ru & \text{ in }\mathbb{R}^+\times[0,T]\\
u(0,t)=0& \text{ for } t\in[0,T]\\ 
u(x,T)=(x-p)^+&\text{ for } x \in\mathbb{R}^+
\end{array}\right.
$$
where $p>0$ represents a given strike price, $\sigma >0$ is the volatility, $r>0$ is the interest rate and 
$$
(x-p)^+=\left\{\begin{array}{ll} x-p & \text{ if
}x>p\\ 0 & \text{ if }x\leq p.
\end{array}\right.
$$

Let $v(x,t)=u(x,T-t)$, then it satisfies the forward Black-Scholes equation, defined for all time $t\in\mathbb{R}^+$ by 
$$
\left\{\begin{array}{ll}
\frac{\partial v}{\partial t}=\frac{1}{2}\sigma^2x^2\frac{\partial^2v}{\partial x^2}+rx\frac{\partial v}{\partial x}-rv & \text{ in
}\mathbb{R}^+\times\mathbb{R}^{+}\\
v(0,t)=0& \text{ for } t\in\mathbb{R}^+\\ 
v(x,0)=(x-p)^+ &\text{ for } x \in\mathbb{R}^+
\end{array}\right.
$$

This problem can be expressed in  an abstract form 
$$
\left\{\begin{array}{ll}
\frac{\partial v}{\partial t}=\mathcal{B}v, & \empty\\
v(0,t)=0, & \empty\\ 
v(x,0)=(x-p)^+ &\text{ for } x \in\mathbb{R}^+, 
\end{array}\right.   
$$
where $\mathcal{B}=(D_\nu)^2+\gamma(D_\nu)-rI$, being $D_{\nu}=\nu x\frac{\partial}{\partial x}$ with  $\nu=\frac{\sigma}{\sqrt{2}}$ and
$\gamma=\frac{r}{\nu}-\nu$.

 In \cite{goldstenmininniromanely} it was shown  that the Black-Scholes equation admits a $C_0$-semigroup solution which can be represented by $T_t:=f(tD_\nu)$, where
 $$
 f(z)=e^{g(z)} \mbox{ with }  g(z)=z^2+\gamma z-r \, , 
 $$
 and a new explicit formula for the solution of the Black-Scholes equation was given in certain spaces of functions
 $Y^{s,\tau}$ defined by
 $$
 Y^{s,\tau}=\left\{u\in C(]0,\infty [) \ ; \ \lim_{x\rightarrow\infty}\frac{u(x)}{1+x^s}=0, \quad \lim_{x\rightarrow 0}\frac{u(x)}{1+x^{-\tau}}=0\right\}, \ \ s,\tau\in\R , 
 $$
 endowed with the norm
 $$
 ||u||_{Y^{s,\tau}}=\sup_{x>0}\biggl|\frac{u(x)}{(1+x^s)(1+x^{-\tau})}\biggr|.
 $$
 Later, it was  proved in \cite{emiradgoldstein12} that the Black-Scholes semigroup is strongly continuous and chaotic for $s>1, \tau \geq 0$ with $s\nu>1$ and
 it was showed in \cite{murillo-arcila_peris2015strong} that  it satisfies the spectral criterion  (Corollary~\ref{generatorDSW}) under the
 same restrictions on the parameters and, therefore,  the
 Black-Scholes semigroup has the SgSP.
\end{example}

Deeper studies of the two previous examples, containing quantitative estimations of the parameters involving the SgSP, are provided in \cite{baranova_bartoll_martinez-gimenez_rodenas0000on,baranova_bartoll_martinez-gimenez_rodenas0000entropy}.

There exist other $C_0$-semigroups related with PDEs which present the SgSP. In fact, the examples given in \cite{murillo-arcila_peris2015strong} in the
context of strong mixing measures, satisfy either the conditions of Proposition~\ref{Agenerator=>SP2} or the spectral criterion in
\cite{desch-schappacher-webb1997hypercyclic} and, therefore, they have the SgSP. The examples  provided in \cite{murillo-arcila_peris2015strong} include  the
semigroup generated by a linear perturbation of the one-dimensional Ornstein-Uhlenbeck operator \cite{conejero_mangino2010hypercyclic}, the solution $C_0$-semigroup  of a partial differential
equation of population dynamics studied by Rudnicki \cite{rudnicki2012chaoticity}, the solution $C_0$-semigroup associated to Banasiak, Lachowicz and
Moszy{\'n}ski models of \emph{birth-and-death} processes
\cite{banasiak_lachowicz_moszynski2007chaotic,banasiak_moszynski2011dynamics}.

\begin{example}[Translation semigroup] 
Let $1\leq p<\infty$  and let $v:\R_+\rightarrow \R$ be a strictly positive locally integrable function, that is, $v$ is measurable with $\int_0^b v(x) \,
dx<\infty$ for all $b>0$. We consider the space of weighted $p$-integrable functions defined as 
$$ 
X=L^p_{v}(\R_+)=\{f:\R_+ \rightarrow \mathbb{K} \ ; \  f
\mbox{ is measurable and } \norm{f}<\infty\},\index{$L^p_{v}(\R_+)$} 
$$ 
where 
$$
\norm{f}=\Big(\int_0^\infty |f(x)|^p v(x)\, dx\Big)^{1/p}.
$$ 
The \textit{translation semigroup} is then given by 
$$ 
T_tf(x)=f(x+t),\quad t,x\geq 0. 
$$ 
This defines a $C_0$-semigroup on $L^p_{v}(\R_+)$ if and only if
there exist $M\geq 1$ and $w\in \R$ such that, for all $t\geq 0$, the following condition 
$$ 
v(x)\leq Me^{wt} v(x+t)  \ \mbox{ for almost all } \ x\geq 0 
$$ 
is satisfied. In that case, $v$ is called an \emph{admissible weight function} and we will assume in the sequel that $v$ belongs to this class of weight
functions.

For the translation semigroup  defined on  $L^p_{v}(\R_+)$, it was proved in  \cite{mangino_peris2011frequently} that $(T_t)_{t\geq 0}$ is Devaney chaotic if and
only if it satisfies the Frequent Hypercyclicity Criterion for semigroups, which in turn is equivalent to the fact  that every operator $T_t$ satisfies
the Frequent Hypercyclicity Criterion for operators. A characterization of frequently hypercyclic shifts on $\ell^p$ due to Bayart and Ruzsa \cite{bayart_ruzsa2015difference} allowed a 
more complete characterization of the frequently  hypercyclic criterion for the translation semigroup  on
$L^p_{v}(\R_+)$ given in \cite{mangino_murillo-arcila2015frequently}:

\begin{theorem}[Theorem 3.10, \cite{mangino_murillo-arcila2015frequently}]\label{mangino_murillo}
Let $v$ be an admissible weight function on $\R^+$. The following assertions are equivalent:
\begin{enumerate}
\item[\rm(1)] The translation semigroup $(T_t)_{t\geq 0}$ is frequently hypercyclic on $L^p_{v}(\R_+)$,

\item[\rm(2)]
$\int_{0}^{\infty}  v(t)dt  < \infty$,

\item[\rm(3)]
 $(T_t)_{t\geq 0}$ is Devaney chaotic on $L^p_{v}(\R_+)$,

\item[\rm(4)]
 $(T_t)_{t\geq 0}$ satisfies the Frequently Hypercyclicity Criterion.
\end{enumerate}
\end{theorem}

This result allows us to give a characterization of the SgSP for  the translation semigroup on the space $X=L^p_{v}(\R_+)$.

\begin{theorem} \label{SP_translations}
Let us consider the translation semigroup on the space $X=L^p_{v}(\R_+)$, where $1\leq p<\infty$ and $v:\R_+\rightarrow \R$ is an admissible weight function. The following are equivalent:

\begin{enumerate}
\item[\rm(i)]
$\int_0^{\infty} v(x) \, dx < \infty$,
\item[\rm(ii)]
$(T_t)_{t\geq 0}$ has SgSP,
\item[\rm(iii)]
$(T_t)_{t\geq 0}$ is Devaney chaotic,
\item[\rm(iv)]
 $(T_t)_{t\geq 0}$ satisfies the Frequently Hypercyclicity Criterion,
 \item[\rm(v)]
 the translation semigroup $(T_t)_{t\geq 0}$ is frequently hypercyclic.
\end{enumerate}
\end{theorem}

\begin{proof}
By Theorem \ref{mangino_murillo} \cite{mangino_murillo-arcila2015frequently} and Propositions \ref{SP->FH} and \ref{FHC->SP2}, it is obvious that  for the
translation semigroup the SgSP is equivalent to satisfy the Frequently Hypercyclicity Criterion and the SgSP is equivalent to frequently hypercyclic.
\end{proof}
\end{example}

\section*{Acknowledgements}

The authors were supported by  MINECO,  Projects MTM2013-47093-P and MTM2016-75963-P. The second and third authors were also supported by Generalitat Valenciana, Projects PROMETEOII/2013/013 and PROMETEO/2017/102.





\end{document}